\documentclass[11pt]{article}
\usepackage{mathrsfs}
\usepackage[all]{xy}
\usepackage{amssymb,latexsym,amsmath,amsthm, mathrsfs}
\usepackage{graphicx}
\usepackage{setspace}
\usepackage{fullpage}

\newtheorem{theorem}{Theorem}[section]
\newtheorem{lemma}[theorem]{Lemma}

\newtheorem{definition}[theorem]{Definition}

\newtheorem{proposition}[theorem]{Proposition}

\newtheorem{remark}[theorem]{Remark}
 \DeclareMathSymbol{\N}{\mathbin}{AMSb}{"4E}

\DeclareMathSymbol{\Z}{\mathbin}{AMSb}{"5A}
\DeclareMathSymbol{\R}{\mathbin}{AMSb}{"52}
\DeclareMathSymbol{\Q}{\mathbin}{AMSb}{"51}
\DeclareMathSymbol{\I}{\mathbin}{AMSb}{"49}
\DeclareMathSymbol{\C}{\mathbin}{AMSb}{"43}

\def\L{{\mathbb L}}

\numberwithin{equation}{section}

\begin{document}

\title{Stringy Mirror Symmetry}

\author{Jyh-Haur Teh}



\date{}

\maketitle

\begin{abstract}
We prove that the mirror pairs constructed by Batyrev and Borisov
have stringy mirror symmetry.
\end{abstract}

\section{Introduction}
The topological mirror symmetry test predicts that if two smooth
$n$-dimensional Calabi-Yau manifolds $V$ and $V^*$ form a mirror
pair, then their Hodge numbers satisfy the relations
$$h^{p, q}(V)=h^{n-p, q}(V^*), 0\leq p, q\leq n$$
But many mirror pairs $(V, V^*)$ found by physicists are Calabi-Yau
varieties with Gorenstein abelian quotient singularities and these
relations do not hold if $V$ or $V^*$ is not smooth. To formulate a
correct mirror symmetry, Batyrev and Dais introduced (see \cite{BD})
the notion of stringy Hodge numbers $h^{p, q}_{st}$ for Calabi-Yau
varieties with at worst log-terminal singularities and Batyrev
modified the topological mirror symmetry test (see \cite{Baty}) to
$$h^{p, q}_{st}(V)=h^{n-p, q}_{st}(V^*), 0\leq p, q\leq n$$
He and Borisov (see \cite{BB1}) proposed a general construction of
mirror pairs which includes mirror constructions of physicists for
rigid Calabi-Yau manifolds, and in \cite{BB2}, they showed that
their mirror pairs satisfy the modified topological mirror symmetry
test.

In \cite{Teh}, we modified the construction of motivic integration
and introduced the notion of stringy invariants. One of the most
interesting invariants to us is given by the dimension of subspaces
of cohomology groups which are generated by algebraic cycles. In
particular, we showed that for two birational Calabi-Yau manifolds,
their corresponding subspaces generated by algebraic cycles have the
same dimensions. As a consequence, if the Hodge conjecture or
Grothendieck standard conjecture is true in one of them, then it is
true for another one. We conjectured (see \cite[Conjecture 3]{Teh})
that the Batyrev-Borisov's mirror pairs satisfy a stringy mirror
symmetry test. The main result in this paper is to show that this
actually follows directly from Batyrev-Borisov's result.

\section{Main result}
We recall briefly notations from \cite{Teh}. Let $K_0(Var)$ be the
Grothendieck group of complex algebraic varieties and
$S=\{\L^i\}_{i\geq 0}$ where $\L=\C$ and we denote by
$\L^0=1=point$. Let $\mathscr{L}=\Z\{\L^i|i\in \Z_{\geq 0}\}$ be the
free abelian group generated by $\L^i$ for all nonnegative $i$. Then
$K_0(Var)$ is canonically a $\mathscr{L}$-module. Let
$\mathcal{N}=S^{-1}K_0(Var)$ be the localization of $K_0(Var)$ with
respect to $S$. Let $F^k\mathcal{N}$ be the subgroup of
$\mathcal{N}$ generated by elements of the form $[X]/\L^i$ where
$i-dim X\geq k$. Then we have a decreasing filtration
$$\cdots \supset F^kN \supset F^{k+1}N \supset \cdots $$
of abelian subgroups of $\mathcal{N}$. The Kontsevich group of
varieties is defined to be
$\hat{\mathcal{N}}:=\underset{\leftarrow}{\lim}\frac{\mathcal{N}}{F^k\mathcal{N}}$
and $\overline{\mathcal{N}}$ is the image of the canonical map
$\mathcal{N} \rightarrow \hat{\mathcal{N}}$.  Then
$\overline{\mathcal{N}}$ is canonically a $\mathscr{L}$-submodule. A
motivic invariant is a group homomorphism from
$\overline{\mathcal{N}}[(\frac{1}{\L^i-1})_{i\geq 1}]$ to $\Z$ and
we say that a family of motivic invariants $\phi=\{\phi_{j, n}|j,
n\in \Z\}$ is of type $(a, b)\in \Z\times \Z$ if $\phi_{j,
n}(X\times \L^k)=\phi_{j-ak, n-bk}(X)$ for any $j, n$ and any
variety $X$. We say that $\phi$ is bounded if $\phi_{j, n}(X)$
vanishes for $|j|, |n|$ large enough, depending on $X$.

\begin{definition}
Suppose that $\phi=\{\phi_{j, n}|j, n\in \Z\}$ is a family of
bounded motivic invariants of type $(a, b)$. We define
$$\phi(X; u, v):=\sum_{j, n}\phi_{j, n}(X)u^jv^n$$ and
$$\phi(\L^{-1};u, v)=(u^av^b)^{-1}\phi(\L^0; u, v)$$ then we have a group homomorphism
$$\phi:\overline{\mathcal{N}}[(\frac{1}{\L^i-1})_{i\geq 1}] \rightarrow \Z[[u, v, (u^av^b)^{-1}]]$$
If $X$ is a normal irreducible algebraic variety with at worst
log-terminal singularities, and $\rho:Y \rightarrow X$ is a
resolution of singularities such that the relative canonical divisor
$D=\sum^r_{i=1}a_iD_i$ has simple normal crossings. Then the stringy
$\phi$-function of type $(a, b)$ associated to $\phi$ is defined to
be
$$\phi^{st}(X; u, v):=\sum_{J\subset I}\phi(D^0_J; u, v)\prod_{j\in J}\frac{u^av^b-1}{(u^av^b)^{a_j+1}-1}$$
where $I=\{1, ..., r\}$ and
$$D_J=\left\{
        \begin{array}{ll}
          \cap_{j\in J}D_j, \mbox{ if }J\neq \emptyset;\\
          X, \mbox{ if } J=\emptyset.
        \end{array}
      \right.
$$
and $D^0_J:=D_j-\cup_{i\neq J}D_i$.
\end{definition}

\begin{remark}
In \cite{Teh}, Definition 13, $\phi(\L^{-1};u, v)$ was defined to be
$(u^av^b)^{-1}$. This is not totally correct since we want to make
it a group homomorphism, we need to modify it to
$(u^av^b)^{-1}\phi(\L^0;u, v)$. Also, we need to modify Conjecture 3
in \cite{Teh} to the following statement which is the main result of
this paper.
\end{remark}

\begin{theorem}\label{main}
If $(V, W)$ is a Batyrev-Borisov's mirror pair and $\phi^{st}$ is a
stringy $\phi$-function of type $(a, b)$, then
$$\phi^{st}(V; u, v)=(-u^a)^n\phi^{st}(W; u^{-1}, v)$$ where $n$ is the dimension of $V$ and $W$.
\end{theorem}

To prove this result, we first proceed as in the proof of
\cite[Theorem 4.3]{Baty} and then use Batyrev-Borisov's result of
the mirror symmetry of stringy $E$-functions of their mirror pairs.

Let $X$ be a normal $d$-dimensional $\Q$-Gorenstein toric variety
associated with a rational polyhedral fan $\Sigma\subset
N_{\R}=N\otimes \R$ where $N$ is a free abelian group of rank $d$.
For a cone $\sigma\in \Sigma$, let $\sigma^0$ be the relative
interior of $\sigma$. The property that $X$ is $\Q$-Gorenstein is
equivalent to the existence of a continuous function
$\varphi_K:N_{\R} \rightarrow \R_{>0}$ satisfying
\begin{enumerate}
\item $\varphi_K(e)=1$, if $e$ is a primitive integral generator of a 1-dimensional cone $\sigma\in \Sigma$.
\item $\varphi_K$ is linear on each cone $\sigma\in \Sigma$.
\end{enumerate}

The following desingularization result can be found in
\cite[Proposition 5-2-2]{KMM}.

\begin{proposition}
Let $\rho:X' \rightarrow X$ be a toric desingularization of $X$,
which is defined by a subdivision $\Sigma'$ of the fan $\Sigma$.
Then the irreducible components $D_1, ..., D_r$ of the exceptional
divisor $D$ of the birational morphism $\rho$ have only normal
crossings and they one-to-one correspond to primitive integral
generators $e'_1, ..., e'_r$ of those 1-dimensional cones
$\sigma'\in \Sigma'$ which do not belong to $\Sigma$. Moreover, in
the formula
$$K_{X'}=\rho^*K_X+\sum^r_{i=1}a_iD_i$$
one has $a_i=\varphi_K(e'_i)-1$ for all $i\in \{1, ..., r\}$.
\end{proposition}

\begin{lemma}
If $\phi^{st}$ is a stringy function of type $(a, b)$, then
$$\phi^{st}(X;u, v)=E_{st}(X; u^a, v^b)\phi(\L^0;u, v)$$
where $E_{st}(X; u, v)$ is the stringy $E$-function of $X$.
\end{lemma}

\begin{proof}
Since the proof follows exactly as the proof \cite[Theorem
4.3]{Baty}, we briefly sketch the main idea of the proof and refer
the reader to the original paper of Batyrev. Let $\Sigma'(J)$ be the
star of the cone $\sigma_J\subset \Sigma'$, i.e., $\Sigma'(J)$
consists of the cones $\sigma'\in \Sigma'$ such that $\sigma'\supset
\sigma_J$. Denote by $\Sigma'_0(J)$ the subfan of $\Sigma'(J)$
consisting of those cones $\sigma'\in \Sigma'(J)$ which do not
contain any $e'_i$ where $i\notin J$. Then the fan $\Sigma'(J)$
defines the toric subvariety $D_J\subset X'$ and the fan
$\Sigma'_0(J)$ defines the open subset $D^0_J\subset D_J$. The
canonical stratification by torus orbits
$$X'=\bigcup_{\sigma'\in \Sigma'}X'_{\sigma}$$ induces the following stratifications
$$D^0_J=\bigcup_{\sigma'\in \Sigma'_0(J)}X'_{\sigma'},  \ \ \ \emptyset \neq J\subset I$$
Then apply the stringy $\phi$-function, we get
$$\phi^{st}(D^0_J;u, v)=\sum_{\sigma'\in \Sigma'_0(J)}(u^av^b-1)^{d-\mbox{dim }\sigma'}\phi(\L^0; u, v)$$
for $\emptyset \neq J\subset I$.

Let $\Sigma'(\emptyset)$ be the subfan of $\Sigma'$ consisting for
those cones $\sigma'\in \Sigma'$ which do not contain any element of
$\{e'_1, ..., e'_r\}$. Then $\Sigma'(\emptyset)$ defines the
canonical stratification of
$$X'\backslash D=\sum_{\sigma'\in \Sigma'(\emptyset)}X'_{\sigma'}$$
Then we have
$$\phi(X'\backslash D;u, v)=\sum_{\sigma'\in \Sigma'(\emptyset)}(u^av^b-1)^{d-\mbox{dim }\sigma'}$$
We may write
$$\sum_{n\in \sigma^0_J\cap N}(u^av^b)^{-\varphi_K(n)}
=\prod_{j\in J}\frac{(u^av^b)^{-\varphi_K(e'_j)}}{1-(u^av^b)^{-\varphi_K(e'_j)}}
=\prod_{j\in J}\frac{1}{(u^av^b)^{a_j+1}-1}$$ Therefore
$$\prod_{j\in J}\frac{u^av^b-1}{(u^av^b)^{a_j+1}-1}=
(u^av^b-1)^{|J|}\sum_{n\in\sigma^0_J\cap
N}(u^av^b)^{-\varphi_K(n)}$$ So by the resolution of singularities $\rho:X' \rightarrow X$, we have
$$\phi^{st}(X; u, v)=\phi(X\backslash D; u, v)+
\sum_{\emptyset\neq J\subset I}\phi(D^0_J;u, v)(u^av^b-1)^{|J|}(\sum_{n\in \sigma^0_J\cap N}(u^av^b)^{-\varphi_K(N)})
$$
$$=\sum_{\sigma'\in \Sigma'(\emptyset)}(u^av^b-1)^{d-\mbox{dim }\sigma'}+
\sum_{\emptyset\neq J\subset I}(\sum_{\sigma'\in
\Sigma'_0(J)}(u^av^b-1)^{d+|J|-\mbox{dim }\sigma'})
(\sum_{n\in\sigma^0_J\cap N}(u^av^b)^{-\varphi_K(n)})$$ By the second part of the computation \cite[Theorem 4.3]{Baty}, we have
$$\bigcup_{\sigma\in \Sigma}\sigma^0\cap N=\bigcup_{\sigma'\in
\Sigma'}(\sigma')^0\cap N$$ This gives us the equality
$$\phi^{st}(X; u, v)=(u^av^b-1)^d \sum_{\sigma\in \Sigma}\sum_{n\in \sigma^0\cap
N}(u^av^b)^{-\varphi_K(n)}\phi(\L^0; u, v)=E_{st}(X; u^a,
v^b)\phi(\L^0; u, v)$$
\end{proof}

\begin{proof}
Now Theorem ~\ref{main} is an easy consequence of the above
computation and Batyrev-Borisov's result. By Batyrev-Borisov's
result, $(V, W)$ satisfies the mirror symmetry $E_{st}(V; u,
v)=(-u)^nE_{st}(W; u^{-1}, v)$. So $\phi^{st}(V; u, v)=E_{st}(V;
u^a, v^b)\phi(\L^0; u, v)=(-u^a)^nE_{st}(W; (u^a)^{-1},
v^b)\phi(\L^0; u, v) =(-u^a)^n\phi^{st}(W; u^{-1}, v)$.
\end{proof}

\bibliographystyle{amsplain}

\end{document}